\documentclass[12pt,reqno]{article}

\usepackage[usenames]{color}
\usepackage{amssymb}
\usepackage{amsmath}
\usepackage{amsthm}
\usepackage{amsfonts}
\usepackage{amscd}
\usepackage{graphicx}
\usepackage{diagbox}
\usepackage{array}
\usepackage{xcolor}
\usepackage[colorlinks=true,
linkcolor=webgreen,
filecolor=webbrown,
citecolor=webgreen]{hyperref}

\usepackage{doi}

\usepackage[backend=biber,style=numeric]{biblatex}
\definecolor{webgreen}{rgb}{0,.5,0}
\definecolor{webbrown}{rgb}{.6,0,0}
\usepackage{url}
\usepackage{caption}
\usepackage{color}
\usepackage{fullpage}
\usepackage{float}

\usepackage{graphics}
\usepackage{latexsym}
\usepackage{epsf}

\usepackage[margin=1in]{geometry} 
\usepackage{tikz}
\usetikzlibrary {arrows.meta}
\usetikzlibrary{decorations.pathmorphing,shapes.geometric,positioning}
\tikzset{note/.style={font=\small}}
\newcommand{\seqnum}[1]{\href{https://oeis.org/#1}{\rm \underline{#1}}}

\theoremstyle{plain}
\newtheorem{theorem}{Theorem}

\newtheorem{lemma}[theorem]{Lemma}
\newtheorem{proposition}[theorem]{Proposition}

\theoremstyle{definition}
\newtheorem{definition}[theorem]{Definition}
\newtheorem{example}[theorem]{Example}

\theoremstyle{remark}

\begin{document}
\begin{center}
\vskip 1cm{\LARGE\bf Some Generalizations of the Bridge and Torch Problem
}
\vskip 1 cm
\large
Thang Pang Ern\\
Department of Mathematics \\
National University of Singapore, 10 Lower Kent Ridge Road, Singapore 119076 \\
\href{mailto:thangpangern@u.nus.edu}{\tt thangpangern@u.nus.edu}
\vskip 1 cm
\large
Gerard Sayson\\
Nanyang Junior College \\
128 Serangoon Avenue 3, Singapore 556111 \\
\href{mailto:gerard@gsn.bz}{\tt gerard@gsn.bz}
\end{center}

\begin{abstract} For the classic bridge and torch problem with crossing times $\left\{1,\ldots,n\right\}$, we derive a closed-form expression for the optimal crossing time $T\left(n\right)$ using a recurrence relation derived from the problem's optimal substructure, thus obtaining \[T\left(n\right)=\frac{n^2}{4}+3n-5+\frac{\left(-1\right)^n-1}{8}\]
which holds for all $n\ge 2$. We generalize the problem to a bridge of capacity 3 and obtain the optimal crossing time \[T_3\left(n\right)=\frac{n^{2}}{6}+2n-\frac{181}{36}+\frac{\left(-1\right)^{n}}{4}-\frac{2}{9}\cos\left(\frac{2n\pi}{3}\right)\]
which holds for all $n\ge 7$. As such, we obtain a new sequence
\seqnum{A392834}
in the On-Line Encyclopedia of Integer Sequences. Lastly, we also explore this problem for star graphs, and see how we can recover some classic identities involving the sum of floor functions.
\end{abstract}

\section{Introduction}
The bridge and torch problem is an interesting scheduling \emph{puzzle}: a group of people must cross a single bridge at night with one torch. At most two people can be on the bridge at any time, the torch must be carried on every crossing, and a pair moves at the speed of its slower member. Assign each person a crossing time; the goal is to move everyone from the left side to the right side in minimum total time (Figure \ref{fig:Visualising the bridge and torch problem}).

This problem has appeared in many forms over the years with various anecdotes attached to it. There are numerous names such as the bridge-crossing puzzle, the four men problem, the flashlight puzzle etc. The oldest reference of this classic problem is a puzzle book by Levmore and Cook \cite{LevmoreCook1981}. The well-known version with four people whose crossing times are $1,2,5,10$ units is perhaps the most common \cite{Rote2002}. Beyond being a fun puzzle, the bridge and torch problem is also a neat example of an optimisation task with constraints and trade-offs. With only a few simple rules, it forces careful thought about order and timing. It has also often been used to introduce ideas from computer science, such as dynamic programming.

In this paper, we extend the bridge and torch problem for crossing times $\left\{1,2,\dots,n\right\}$ to the capacity $c = 3$ case and develop steps for further generalizations. We also move beyond the simple linear `left bank to right bank' model to explore the problem on star graphs $\mathcal{S}(k)$
for $k$ leaves, allowing for multiple torches.

\begin{figure}[H]
\centering
\begin{tikzpicture}[
  scale=1,
  person/.style={circle, draw, fill=white, minimum size=6mm, inner sep=0pt, font=\sffamily},
  bank/.style={fill=green!18, draw=green!40!black},
  water/.style={fill=blue!10, draw=blue!40!black},
  bridge/.style={fill=gray!25, draw=gray!55}
]

\draw[water] (2,0) rectangle (12,6);

\draw[bank] (0,0) rectangle (2,6);
\draw[bank] (12,0) rectangle (14,6);

\node at (1,6.35) {Left of bridge};
\node at (13,6.35) {Right of bridge};

\draw[bridge, line join=round] (2,2.5) rectangle (12,3.5);
\foreach \x in {3,4,5,6,7,8,9,10,11}{
  \draw[gray!60] (\x,2.5) -- (\x,3.5);
}
\draw[gray!70,line cap=round] (2,3.65) -- (12,3.65);
\draw[gray!70,line cap=round] (2,2.35) -- (12,2.35);

\foreach \y in {1.0, 1.8, 4.2, 5.0}{
  \draw[blue!45,decorate,decoration={snake,segment length=12mm, amplitude=0.6mm}] (2,\y) -- (12,\y);
}

\def\yA{5.1}
\def\yB{4.0}
\def\yC{2.9}
\def\yD{1.8}

\node[person] (p1) at (1,\yA) {1};
\node[person] (p2) at (1,\yB) {2};
\node[person] (p3) at (1,\yC) {3};
\node[person] (p4) at (1,\yD) {4};

\node[star,star points=5,star point ratio=2,fill=yellow!85!orange,draw=orange!70,minimum size=5mm] (torch) at (0.5,0.7) {};
\node[note,anchor=west] at (0.8,0.7) {torch};

\foreach \y in {\yA,\yB,\yC,\yD}{
  \node[person,draw=gray!50,fill=gray!10] at (13,\y) {};
}

\draw[-{Latex[length=2.4mm]},very thick,gray!60]
  (2.2,3.0) -- (11.8,3.0);
\node[note] at (8.7,3.25) {crossing path};

\node[person,minimum size=5mm] at (2.6,0.4) {};
\node[note,anchor=west] at (2.9,0.4) {node label = crossing time};

\end{tikzpicture}
\caption{Visualising the bridge and torch problem}
\label{fig:Visualising the bridge and torch problem}
\end{figure}
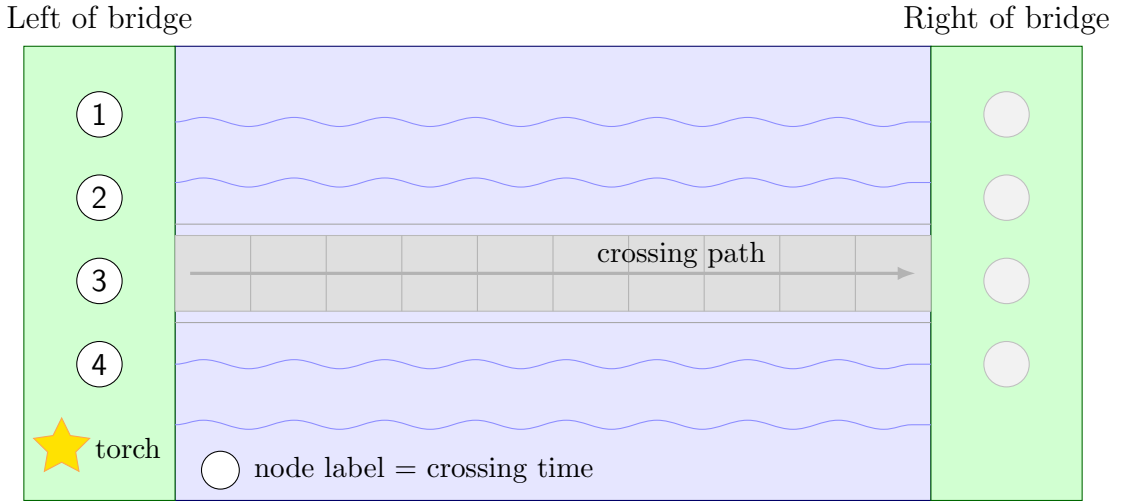
\section{The capacity-two case}
\begin{theorem}\label{theorem one bridge one torch with crossing times 1 to n}
Let there be $n$ persons on one side of the bridge with crossing times $\left\{1,\ldots,n\right\}$. Define the optimal crossing time to be $T\left(n\right)$. Then, \begin{align}\label{theorem: optimal crossing time}
T\left(n\right)=\frac{n^2}{4}+3n-5+\frac{\left(-1\right)^n-1}{8}\quad\text{for all positive integers }n\ge 2.
\end{align}
\end{theorem}
The terms $T\left(1\right),T\left(2\right),T\left(3\right),T\left(4\right),\ldots=1,2,6,11,\ldots$ appear in sequence \seqnum{A078476}. It denotes the set of times taken to get $n$ people from one side of a bridge to the other where (a) the only \emph{flashlight} must be carried when crossing; (b) only one or two people may cross at the same time; (c) a pair crosses at the speed of the slowest member; and (d) the $k^\text{th}$ person's speed requires $k$ units to cross the bridge, where $1\le k \le n$. This sequence is a special case of \seqnum{A318271}, which denotes the optimal crossing time for the bridge and torch problem given that the crossing times for the group's members are given by the $n^\text{th}$ partition of the triangle in which $n^\text{th}$ row lists juxtaposed lexicographically ordered partitions of $n$.

Also, note that equivalently, for all $n\ge 2$,
\begin{align}\label{eqn:formula for T(n)}
    T\left(n\right)=\left\lfloor \frac{n^2+12n-20}{4}\right\rfloor.
\end{align}

Rote proved a generalized version of Theorem \ref{theorem one bridge one torch with crossing times 1 to n} for arbitrary ordered crossing times $0<t_1<\cdots<t_N$ (the present paper treats the special case $t_i=i$). His optimality proof proceeds by translating any feasible schedule into a multigraph optimisation problem, formulating the total crossing time as a weighted objective depending on the multiplicities and degrees of edges, and then classifying the structure of optimal graphs. Finally, he shows that every such optimal graph can be realized by an explicit feasible schedule \cite{Rote2002}. Recently, Jianu, Jianu and Popescu gave a combinatorial proof of the generalized problem involving arbitrary ordered crossing times \cite{JianuJianuPopescu2020}. In contrast, we work with what is known as \emph{canonical states} and derive a direct recurrence that reduces the $n$-person instance to an $\left(n-2\right)$-person instance with an exact additional cost. This yields an optimality proof for the case $t_i=i$. 

Before proving Theorem \ref{theorem one bridge one torch with crossing times 1 to n}, we state Lemma \ref{lemma: odd even sequence general term}, which provides the general formula of a sequence that is defined by separate formulae for even and odd terms. 
\begin{lemma}\label{lemma: odd even sequence general term}
For any $f,g:\mathbb{N}\to\mathbb{R}$, let $\left\{a_n\right\}_{n\ge 0}$ be a sequence defined by separate formulae for its even and odd terms. That is, $a_{2n}=f\left(n\right)$ and $a_{2n+1}=g\left(n\right)$. Then, the general term $a_n$ can be expressed as a single formula as follows: \[a_n=f\left(\frac{n}{2}\right)\frac{1+\left(-1\right)^n}{2}+g\left(\frac{n-1}{2}\right)\frac{1-\left(-1\right)^n}{2}\]
\end{lemma}
One can prove Lemma \ref{lemma: odd even sequence general term} by construction. The idea is to express $a_n$ as a linear combination using characteristic functions for the parity of $n$. That is,
\[
a_n = \mathbf{1}_{\text{even}}\left(n\right) \cdot \text{term for even } n + \mathbf{1}_{\text{odd}}\left(n\right) \cdot \text{term for odd } n,
\]
and utilise the property that $\left(-1\right)^n$ oscillates between $1$ and $-1$.

We now give a proof of Theorem \ref{theorem one bridge one torch with crossing times 1 to n}.
\begin{proof} Let the people be $\left\{1,\ldots,n\right\}$ with crossing times $1<2<\cdots<n$. A state is a pair $\left(L,\lambda\right)$ where $L\subseteq\left\{1,\ldots,n\right\}$ is the set currently on the left bank, and $\lambda\in\left\{\textsf{L},\textsf{R}\right\}$ is the torch position. The initial and terminal states are $s=\left(\left\{1,\ldots,n\right\},\textsf{L}\right)$ and $t=\left(\emptyset,\textsf{R}\right)$ respectively. A move from $\left(L,\textsf{L}\right)$ to $\left(L\setminus A,\textsf{R}\right)$ consists of choosing a set $A\subseteq L$ with $\left|A\right|\in\left\{1,2\right\}$ and sending those people across, so its cost is $\max A$. A move from $\left(L,\textsf{R}\right)$ to $\left(L\cup A,\textsf{L}\right)$ consists of choosing $A\subseteq L'$ with $\left|A\right|\in \left\{1,2\right\}$ and sending those people back, so again its cost is $\max A$. Here, $L'$ denotes the complement of $L$.

For any state $u$, let $\operatorname{OPT}\left(u\right)$ denote the minimum total cost of any valid sequence of moves from $u$ to $t$. For $k\ge 2$, define the \emph{canonical} state $S_k=\left(\left\{1,\ldots,k\right\},\textsf{L}\right)$, so $T\left(k\right)=\operatorname{OPT}\left(S_k\right)$ which agrees with the original definition of $T\left(n\right)$. 

We claim that \begin{align}\label{eqn: T(k) = T(k-2) + k + 5}
T\left(k\right)=T\left(k-2\right)+k+5\quad\text{for all }k\ge 4.
\end{align}
To obtain this recurrence relation, we first consider any optimal plan starting from the canonical state $S_k$. In any plan that begins at $S_k$ and ends at $t$, persons $k-1$ and $k$ must eventually end up on the right bank. Since they start on the left, there is a first occurrence at which both $k-1$ and $k$ are on the right simultaneously. Let this moment occur immediately after some $\textsf{L}\to\textsf{R}$ move. During this move, at least one of $k-1$ and $k$ crosses from left to right. The cost of this move is $\ge k$ as person $k$ must cross at the same time. So, the move that first places both $k-1$ and $k$ on the right has cost $\ge k$.

We now show that in an optimal plan, this move is precisely the crossing $\left\{k-1,k\right\}$ together (so the cost is $k$), and that immediately after it, $\lambda\to\textsf{L}$ by person 1. Indeed, once both $k-1$ and $k$ are on the right, if the the process has not terminated yet, the torch must return to the left at least once more as there are still some people remaining to be transported. Any return move has cost $\ge 1$, and if person 1 is on the right at that time, then using 1 as the returner is never more expensive than using anyone else. Thus, in an optimal plan, we may choose person 1 as the returner at the first return after $k-1$ and $k$ are together on the right. This return costs exactly 1.

So, to obtain the pattern where we get $1\to\textsf{R}$ while keeping $\lambda=\textsf{L}$, then send $\left\{k-1,k\right\}\to\textsf{R}$, and then $1\to\textsf{L}$, we must ensure that person 1 is on the right and $\lambda= \textsf{L}$ immediately before $k-1$ and $k$ cross so that 1 can return after that crossing. 

Starting from $S_k$, suppose at some time, we reach a state with $\lambda=\textsf{L}$ and with person 1 on the right. Since person 1 starts on the left, there must have been at least one $\textsf{L}\to\textsf{R}$ crossing that included 1. Thereafter, since $\lambda=\textsf{L}$ at the moment under consideration, there must have been at least one $\textsf{R}\to\textsf{L}$ return made by someone other than 1, otherwise 1 would have returned to the left. The cheapest way to do this is as follows:
\begin{itemize}
\item The first $\textsf{L}\to\textsf{R}$ move that puts $1$ on the right has cost $\ge 2$ because if $1$ crosses alone, then the only person on the right is $1$, and the only possible returner is $1$, contradicting the requirement that $1$ stays on the right while $\lambda\to \textsf{L}$. Hence at least two people must cross with $1$, and the smallest possible cost is achieved by sending 1 and 2.
\item After 1 and 2 cross, to have the torch back on the left while keeping $1$ on the right, someone other than $1$ must return; the cheapest such returner is $2$, costing $2$.
\end{itemize}
Therefore any plan must incur cost $\ge 2+2=4$ before it can be in a state with $\lambda=\textsf{L}$ and person $1$ on the right. Moreover, this cost $4$ is achieved by the two moves $\left\{1,2\right\}\to \textsf{R}$ with cost 2 and $2\to\textsf{L}$ with cost 2, leaving $1$ on the right and $\lambda=\textsf{L}$.

From the preceding steps, any optimal plan starting at $S_k$ necessarily contains, in some order consistent with feasibility, the following components:
\begin{itemize}
    \item cost 4 to place 1 on the right with $\lambda=\textsf{L}$
    \item an $\textsf{L}\to\textsf{R}$ crossing that first places both $k-1$ and $k$ on the right with cost $\ge k$
    \item at least one return after that costing at least 1
\end{itemize}
Thus, the total cost incurred by the time we have transported $k-1$ and $k$ to the right and brought the torch back to the left is at least $k+5$. On the other hand, the set of four-move sequences $\left\{1,2\right\}\to\textsf{R}$, $2\to\textsf{L}$, $\left\{k-1,k\right\}\to \textsf{R}$, and $1\to \textsf{L}$ is valid and has total cost $k+5$. After these moves, $k-1$ and $k$ would have been transferred to the right, and the system is in the canonical state $S_{k-2}=\left(\left\{1,\ldots,k-2\right\},\textsf{L}\right)$.

From this point onward, only the people $\left\{1,\ldots,k-2\right\}$ remain to be transported in exactly the same problem form. By optimal substructure, the minimum remaining time is $T\left(k-2\right)$. Therefore, we obtain the recurrence relation (\ref{eqn: T(k) = T(k-2) + k + 5}) as claimed.

For the base cases, one can directly check that $T\left(2\right)=2$ and $T\left(3\right)=6$. We then solve the recurrence relation. If $n=2m$ with $m\ge 1$, then \[T\left(2m\right)=T\left(2\right)+\sum_{j=2}^{m}\left(2j+5\right)=m^2+6m-5=\frac{n^2}{4}+3n-5.\]
On the other hand, if $n=2m+1$ with $m\ge 1$, then \[T\left(2m+1\right)=T\left(3\right)+\sum_{j=2}^{m}\left(\left(2j+1\right)+5\right)=\frac{n^2}{4}+3n-5-\frac{1}{4}.\]
By Lemma \ref{lemma: odd even sequence general term}, we combine the parity cases to obtain the desired exact formula for $T\left(n\right)$, which holds for all $n\ge 2$.
\end{proof}
One notes that the sequence of steps for $\left\{1,2,\ldots,n\right\}$ is not unique. For example, in the capacity $c=2$ case, although $T\left(4\right)=11$, the code has the following output (Figure \ref{fig:c2 n4}).

\begin{figure}[H]
    \centering
    \begin{tabular}{|c|c|c|c|c|}
        \hline
        Step & Direction & People & Time & Cumulative \\
        \hline
        1 & \textsf{L} to \textsf{R} & 1, 2 & 2 & 2 \\
        \hline
        2 & \textsf{R} to \textsf{L} & 1 & 1 & 3 \\
        \hline
        3 & \textsf{L} to \textsf{R} & 1, 3 & 3 & 6 \\
        \hline
        4 & \textsf{R} to \textsf{L} & 1 & 1 & 7 \\
        \hline
        5 & \textsf{L} to \textsf{R} & 1, 4 & 4 & 11 \\
        \hline
    \end{tabular}
    \caption{The steps taken for $T(4) = 11$. Time taken was approx 102 $\mu$s.}
    \label{fig:c2 n4}
\end{figure}
As mentioned in Theorem \ref{theorem one bridge one torch with crossing times 1 to n} and sequence \seqnum{A078476}, indeed $T(4)=11$ but we cannot apply this naive strategy of sending 1 to and fro in order to obtain the optimal solution. For example, if we employ this strategy for $T(5)$, we would see that we obtain 17 instead of the supposed optimal time of 16. Let $a(n)$ denote the minimum number of times person 1 travels either across or back. Then, one can easily deduce that $a\left(4\right)=3$ instead of 5. This shows that the sequence of steps in helping $\left\{1,\ldots,n\right\}$ get across the bridge is not unique. Moreover, we have an identity showing an example of the non-uniqueness of the sequence of steps for the case when $n=4$, which is \[2+1+3+1+4=2+1+4+2+2,\]
where the order in which we sum the terms matters. Proposition \ref{proposition:a(n) nice odd} shows a nice connection between $a(n)$ and the sequence of odd numbers repeated.
\begin{proposition}\label{proposition:a(n) nice odd}
Let $a\left(n\right)$ denote the minimum number of times person 1 travels either across or back for the capacity $c=2$ case. Then, $a\left(n\right)$ is equal to
sequence \seqnum{A109613}, which denotes the sequence of odd numbers repeated. The first few terms are \[1, 1, 3, 3, 5, 5, 7, 7, 9, 9,\ldots.\]
\end{proposition}
\begin{proof}
By considering the formula for $T\left(n\right)$ as in Theorem \ref{theorem one bridge one torch with crossing times 1 to n}, one can deduce that the optimal time taken involves considering \emph{blocks} of the form $\left\{1,2\right\}\to \textsf{R}$, $1\to \textsf{L}$, $\left\{k-1,k\right\}\to \textsf{R}$, $2\to \textsf{L}$. If $n\ge 4$ is even, then there are $\frac{n}{2}-1$ such blocks, followed by a single move where $\left\{1,2\right\}\to \textsf{R}$; if $n\ge 5$ is odd, then there are $\frac{n-3}{2}$ such blocks, followed by the moves $\left\{1,2\right\}\to \textsf{R}$, $1\to \textsf{L}$, $\left\{1,3\right\}\to\textsf{R}$. In each case, one can count the number of times person 1 travels either across or back, which indeed yields
the desired sequence.
\end{proof}
We can visualise the set of moves taken to obtain $T\left(n\right)$. We give two examples for $T\left(5\right)$ and $T\left(8\right)$ as shown in Figures \ref{fig:visualising set of crossings for T(5)} and \ref{fig:visualising set of crossings for T(8)} respectively. Here, each edge denotes a move made by the two vertices, where vertex $i$ represents person $i$ with crossing time $i$. Unlike Rote's visualisation in \cite{Rote2002}, ours incorporates \emph{loops} for vertices 1 and 2 so that Proposition \ref{proposition:a(n) nice odd} becomes more apparent --- $a\left(n\right)$ counts the degree of vertex 1. Moreover, we have colored vertices 1 and 2 as they have more interaction than the others.
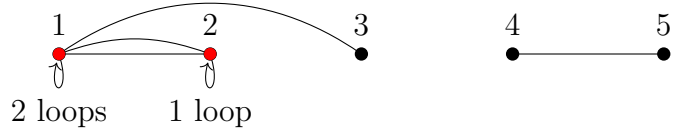
\begin{figure}[H]
    \centering
    \begin{tikzpicture}[scale=1]

    \foreach \x/\lab in {0/1,2/2,4/3,6/4,8/5}{
      \fill (\x,0) circle (2.5pt);
      \node[above] at (\x,0.1) {\lab};
    }
    
    \draw[bend left=20] (0,0) to (2,0);
    \draw[bend left=35] (0,0) to (4,0);
    
    \draw (6,0) to (8,0);
    \draw (0,0) to (2,0);

    \node[circle,draw,inner sep=1.5pt, draw=red, fill=red] (v1) at (0,0) {};
    \draw (v1) edge[loop below] node[below]{2 loops} ();

    \node[circle,draw,inner sep=1.5pt, draw=red, fill=red] (v2) at (2,0) {};
    \draw (v2) edge[loop below] node[below]{1 loop} ();
    
    \end{tikzpicture}
    \caption{Visualising the set of moves to obtain $T\left(5\right)$}
    \label{fig:visualising set of crossings for T(5)}
\end{figure}

\begin{figure}[H]
    \centering
    \begin{tikzpicture}[scale=1]

    \foreach \x/\lab in {0/1,2/2,4/3,6/4,8/5,10/6,12/7,14/8}{
      \fill (\x,0) circle (2.5pt);
      \node[above] at (\x,0.1) {\lab};
    }
    
    \draw[bend left=20] (0,0) to (2,0);
    \draw[bend left=35] (0,0) to (2,0);
    \draw[bend left=55] (0,0) to (2,0);
    \draw (0,0) to (2,0);

    \draw (4,0) to (6,0);
    \draw (8,0) to (10,0);
    \draw (12,0) to (14,0);

    \node[circle,draw,inner sep=1.5pt, draw=red, fill=red] (v1) at (0,0) {};
    \draw (v1) edge[loop below] node[below]{3 loops} ();

    \node[circle,draw,inner sep=1.5pt, draw=red, fill=red] (v2) at (2,0) {};
    \draw (v2) edge[loop below] node[below]{3 loops} ();
    
    \end{tikzpicture}
    \caption{Visualizing the set of moves to obtain $T\left(8\right)$}
    \label{fig:visualising set of crossings for T(8)}
\end{figure}
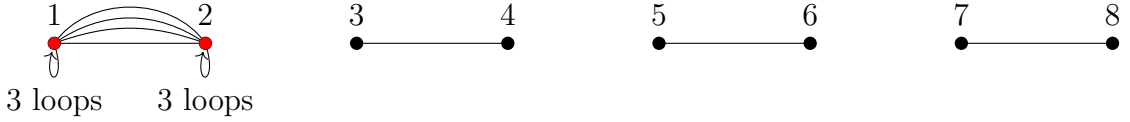

\section{The capacity-three case}
Let $c$ denote the capacity of the bridge, meaning that at most $c$ persons may be on the bridge in a single move. The classical bridge and torch puzzle corresponds to the case $c=2$, which was analysed in Theorem \ref{theorem one bridge one torch with crossing times 1 to n}. In this section, we study the bridge and torch problem for capacity $c=3$, while retaining the single-torch constraint. Our objective is to describe the structural features of optimal schedules. Although Backhouse and Truong \cite{BackhouseTruong2015} analysed the general capacity-$c$ torch problem using dynamic programming, our closed formula considers the case where crossing times are $\left\{1,\ldots,n\right\}$ for $c=2,3$.

\begin{theorem}\label{theorem capacity c optimal}
Let $T_3\left(n\right)$ denote the minimum total time to transport persons $\left\{1,2,\dots,n\right\}$ with crossing times $1<2<\cdots<n$ across a bridge of capacity 3 with a single torch, where each move carries the torch and the time of a move equals the maximum crossing time among the movers. For all $n\ge 7$, we have
\begin{align*}
    T_3\left(n\right)=\frac{n^{2}}{6}+2n-\frac{181}{36}+\frac{\left(-1\right)^{n}}{4}-\frac{2}{9}\cos\left(\frac{2n\pi}{3}\right)
\end{align*}
with initial conditions $T_3\left(1\right)=1$, $T_3\left(2\right)=2$, $T_3\left(3\right)=3$, $T_3\left(4\right)=7$, $T_3\left(5\right)=9$, $T_3\left(6\right)=14$.
\end{theorem}
We will prove Theorem \ref{theorem capacity c optimal}
by ``squeezing'' $T_3(n)$ between two bounds to form a recurrence relation,
then solving that recurrence relation.

   

\begin{proposition}\label{proposition initial val t3}
    Keep the notation in Theorem \ref{theorem capacity c optimal}.
    Then
    $T_3\left(1\right)=1$, $T_3\left(2\right)=2$, $T_3\left(3\right)=3$, $T_3\left(4\right)=7$, $T_3\left(5\right)=9$, $T_3\left(6\right)=14$.
\end{proposition}
\begin{proof}
    By complete search.
\end{proof}

Before we can make any further combinatorial
arguments to prove Proposition \ref{proposition recurrence t3}
we will prove a few useful lemmas.
\begin{lemma}\label{lemma fastest person always returns}
    Only the fastest person on $\mathsf{R}$ makes a return trip
    to $\mathsf{L}$.
\end{lemma}
\begin{proof}
    See Sniedovich \cite{Sniedovich2002}.
\end{proof}

\begin{lemma}\label{lemma fastest person ct}
    In any optimal schedule,
    the fastest person that makes a required return trip has crossing time
    no more than $4$ units.
\end{lemma}
\begin{proof}
    Consider an optimal schedule. By Lemma \ref{lemma fastest person always returns},
    whenever a return from $\mathsf{R}$ to $\mathsf{L}$ is required,
    the returning person must be the fastest person
    present on $\mathsf{R}$.
    
    Suppose, for contradiction, that at some such return the fastest person on $\mathsf{R}$
    has crossing time $t_r \ge 5$.
    In order for this person to be available on $\mathsf{R}$ and then to return,
    they must have crossed from $\mathsf{L}$ to $\mathsf{R}$ at some earlier time
    and now cross back, incurring a total cost of at least $2t_r \ge 10$.
    
    We show that this situation cannot occur in an optimal schedule.
    Instead of allowing the person with crossing time $t_r$ to be the fastest available
    returner, we may stage faster support earlier at strictly lower cost.
    Specifically, send persons $1$, $2$ and $3$ from $\mathsf{L}$ to $\mathsf{R}$,
    then return person $1$ to $\mathsf{L}$.
    This leaves person $2$ and $3$ on $\mathsf{R}$ with the torch on $\mathsf{L}$,
    having incurred total cost $4$.
    
    After this modification, whenever a return is required,
    person $2$ is available to return the torch at cost $2$,
    which is strictly cheaper than returning a person with cost $t_r \ge 5$.
    Thus any schedule in which the fastest available returner has crossing time at least $5$
    can be transformed into one of strictly lower total cost.
    
    This contradicts optimality. Hence, in any optimal schedule,
    whenever a return from $\mathsf{R}$ to $\mathsf{L}$ is required,
    the fastest person available to return has crossing time at most $4$.
\end{proof}

We are now in a position to prove a recurrence relation.

\begin{proposition}\label{proposition recurrence t3}
    For $n \ge 13$, we have $T_3(n) = T_3(n-6) + 2n + 6$.
\end{proposition}
\begin{proof}
    We first prove that the upper bound $T_3\left(k\right)\le T_3\left(k-6\right)+2k+6$ holds by constructing a sequence of moves
    from $S_k$ to $S_{k-6}$.
    Starting from $S_k$, consider the valid sequence of moves in
    Figure \ref{fig:c3 block schedule}.
    \begin{figure}[ht!]
        \centering
        \begin{tabular}{|c|c|c|c|c|}
            \hline
            Step & Direction & People & Time & Cumulative \\
            \hline
            1 & \textsf{L} to \textsf{R} & 1, 2, 3 & 3 & 3 \\
            \hline
            2 & \textsf{R} to \textsf{L} & 1 & 1 & 4 \\
            \hline
            3 & \textsf{L} to \textsf{R} & $k-2$, $k-1$, $k$ & $k$ & $k+4$ \\
            \hline
            4 & \textsf{R} to \textsf{L} & 2 & 2 & $k+6$ \\
            \hline
            5 & \textsf{L} to \textsf{R} & $k-5$, $k-4$, $k-3$ & $k-3$ & $2k+3$ \\
            \hline
            6 & \textsf{R} to \textsf{L} & 3 & 3 & $2k+6$ \\
            \hline
        \end{tabular}
        \caption{A sequence of moves from $S_k$ to $S_{k-6}$.}
        \label{fig:c3 block schedule}
    \end{figure}
    After these moves, exactly the six people $\{k-5,k-4,k-3,k-2,k-1,k\}$ are on the right bank, while
    $\{1,\ldots,k-6\}$ remain on the left bank and the torch is on the left; i.e. we have reached $S_{k-6}$. The total cost of these six moves is $3+1+k+2+(k-3)+3=2k+6$. Any optimal solution must cost no more than $2k+6$ units of time, establishing the upper bound $T_3(k)\le (2k+6)+T_3(k-6)$.

    We now prove a lower bound on the cost incurred by returns that must be performed 
    without using the six slowest persons. This is possible due to Lemma \ref{lemma fastest person ct}, which ensures that in an optimal schedule, whenever a return of 
    the torch is required, the fastest available returner has crossing time at most~$4$.

    To reduce the state from $S_k$ to $S_{k-6}$, there must be two expensive 
    left-to-right moves, of cost at least $k$ and $k-3$, respectively. Each such move 
    necessarily places the torch on the right bank while the process is not yet 
    finished, and hence must be followed by a return. Moreover, after the first such 
    return, at least one person must remain on the right bank in order to permit a 
    further return following the second expensive move. Consequently, between these two 
    expensive left-to-right moves, there must be three distinct return trips.
    
    By Lemma \ref{lemma fastest person ct}, each of these returns must be carried out by 
    a person of crossing time at most~$4$. We now show that, regardless of how the 
    corresponding support persons are staged, these three returns incur a total cost of 
    at least $4+2+3$.
    
    Indeed, consider the earliest time at which the torch is on the left bank and at 
    least two persons are present on the right bank. Achieving such a configuration 
    costs at least~$4$: at some earlier point, at least three persons must have crossed 
    to the right in a single move (otherwise, after the necessary return that places the 
    torch back on the left, at most one person would remain on the right). Any such move 
    costs at least~$3$, and the subsequent return costs at least~$1$.
    
    After the first expensive left-to-right move, the torch must return to the left bank while leaving at least one support person on the right. The minimum possible cost of such a return is~$2$. After the second expensive left-to-right move, a further return is required; at this point, at least two fast persons have already been used for earlier returns, and hence the cheapest remaining support returner has crossing time at least~$3$.
    
    Therefore, irrespective of the timing of support staging, the total cost incurred by the necessary return operations is at least $4+2+3=9$. Putting these unavoidable contributions together, before the process can reduce from $S_k$ to $S_{k-6}$ it must incur total cost at least $k+(k-3)+9 = 2k+6$. Hence equality holds and the recurrence follows.
\end{proof}

A direct check (by exhaustive search on the finite state space) gives
\[
T_3(7)=17,\quad T_3(8)=22,\quad T_3(9)=26,\quad T_3(10)=32,\quad T_3(11)=37,\quad T_3(12)=43.
\]
Fix $r\in\{0,1,2,3,4,5\}$ and write $n=6q+r$ with $q\ge 1$ and $n\ge 7$.
Iterating Proposition \ref{proposition recurrence t3} along the arithmetic progression $6q+r,\,6(q-1)+r,\,\ldots$ yields a quadratic in $q$,
particularly $T_3(6q+r)=6q^2+(2r+12)q+c_r$ for suitable constants $c_r$ determined by the above base block.
Equivalently (with $r$ interpreted in $\{0,\ldots,5\}$) one obtains the piecewise formula in residue classes mod $6$.

We are now in a position to prove Theorem \ref{theorem capacity c optimal}
by finding a closed formula encoding Propositions \ref{proposition initial val t3}
and \ref{proposition recurrence t3}. Define
\[
E(n)=T_3(n)-\frac{n^2}{6}-2n.
\]
Using Proposition \ref{proposition recurrence t3} and the identity
\[
\left(\frac{n^2}{6}+2n\right)-\left(\frac{(n-6)^2}{6}+2(n-6)\right)=2n+6,
\]
we get $E(n)=E(n-6)$ for all $n\ge 13$, hence $E(n)$ is $6$-periodic on $n\ge 7$. Therefore, $E(n)$ is determined by its values on $n=7,8,9,10,11,12$. Since $E(n)$ has period $6$, it can be expressed using a period-$2$ term and a period-$3$ term as follows:
\[
E(n)=A+B(-1)^n+C\cos\left(\frac{2\pi n}{3}\right)\quad\text{where }A,B,C\in\mathbb{R}.
\]
Substituting $n=7,8,9$ and solving the resulting $3\times 3$ linear system gives $A=-\frac{181}{36}$, $B=\frac{1}{4}$, and $C=-\frac{2}{9}$. Hence for all $n\ge 7$,
\[T_3(n)=\frac{n^2}{6}+2n-\frac{181}{36}+\frac{(-1)^n}{4}-\frac{2}{9}\cos\!\left(\frac{2\pi n}{3}\right),\]
proving Theorem \ref{theorem capacity c optimal}. The values $T_3(1)=1,\ldots,T_3(6)=14$ are as stated in Proposition \ref{proposition initial val t3}.
\begin{example}[$c=3$] When $c=3$, applying Theorem \ref{theorem capacity c optimal}
generates the sequence of values for $T_3(n)$ as follows:
\[
\begin{array}{c|cccccccccccccccccccc}
n&1&2&3&4&5&6&7&8&9&10&11&12&13&14&15&16&17&18&19&20\\\hline
T_3\left(n\right)&1&2&3&7&9&14&17&22&26&32&37&43&49&56&62&70&77&85&93&102
\end{array}
\]
\end{example}

\begin{example}[$c=4$] When $c=4$, by complete
search one generates the sequence
\[
\begin{array}{c|cccccccccccccccccccc}
n&1&2&3&4&5&6&7&8&9&10&11&12&13&14&15&16&17&18&19&20\\\hline
T_4\left(n\right)&1&2&3&4&8&10&12&17&20&23&28&32&36&42&47&52&58&64&70&77
\end{array}
\]
\end{example}
Note that $T_c\left(n\right)$ is most interesting for small capacities $c$. Indeed, as $c$ grows relative to $n$, the constraint becomes weak: if $c\ge n$, everyone can cross in a single move, so $T_c\left(n\right)=n$.
It is most likely difficult to study a formula for $c = 4$ and beyond, especially
given the complex nature of the recurrence derived. However we conjecture
that, for small $c$, there is a closed formula or at least a recurrence
for $T_c(n)$.

\section{A generalization to star graphs}
As we have seen, the bridge and torch problem usually involves a single bridge with limited capacity. The task is to decide who crosses together, who returns with the torch, and in what order, to finish as soon as possible. Now, consider a setting with more than one bridge. At a node, some travelers go in different direction. The goal is to schedule movements so that everyone reaches a destination quickly.

Here, we will discuss some applications to star graphs (Definition \ref{definition of star graph}). In such a graph, travellers typically start at the root and must reach the leaves.
\begin{definition}[star graph]\label{definition of star graph} Fix an integer $k\ge 1$. The star graph with $k$ edges, denoted by $\mathcal{S}(k)$, is the graph with vertex set $V=\{s\}\cup\{u_1,\ldots,u_k\}$ where $s$ is the center and $u_1,\ldots,u_k$ are the leaves. We define the edge set to be $E=\{\{s,u_i\}:i=1,\ldots,k\}$. Equivalently, $S(k)$ is the unique graph on $k+1$ vertices in which $\deg(s)=k$ and
$\deg(u_i)=1$ for all $i=1,\dots,k$, up to graph isomorphism.
\end{definition}
Essentially, one can think of a star graph as a hub-and-spoke system where there is a single root node and multiple child nodes connected only to the root node. We will work on star graphs up to rotation. For example, Figure \ref{fig:star graph with 7 edges} shows a star graph with root/center $s$ and edges $\left\{s,u_i\right\}$ where $i=1,\ldots,7$.
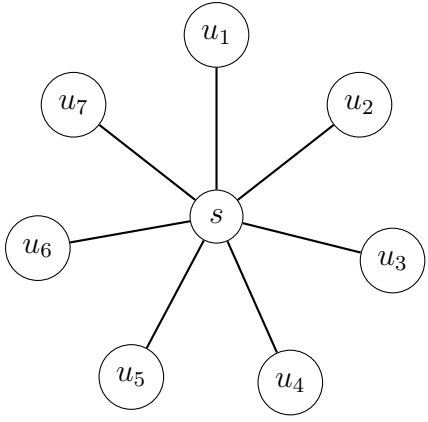
\begin{figure}[H]
    \centering
\begin{tikzpicture}[
    every node/.style={circle,draw,fill=white,minimum size=7mm,font=\sffamily},
    edge/.style={draw,thick},scale=0.8
]

\node (s) at (0,0) {$s$};

\foreach \i/\angle in {1/90,2/38,3/-14,4/-66,5/-118,6/-170,7/142} {
    \node (u\i) at (\angle:3cm) {$u_{\i}$};
    \draw[edge] (s) -- (u\i);
}

\end{tikzpicture}
    \caption{The star graph $\mathcal{S}(7)$ with 7 edges}
    \label{fig:star graph with 7 edges}
\end{figure}

We now formally define the problem for star graphs. 
\begin{definition}[parallel step constraints on a star] For the bridge and torch problem on star graphs, a \emph{parallel step} is a collection of torch-movements executed simultaneously,
satisfying the following conditions:
\begin{itemize}
    \item Disjoint people: no person participates in more than one torch-movement in that step
    \item Disjoint star edges: no star edge is traversed by more than one torch-movement in that step
\end{itemize}
The duration of a parallel step is the maximum duration among its torch-moves.
\end{definition}

\begin{theorem}[bridge and torch problem for star graphs]\label{theorem: bridge and torch for star graphs}
Let $T(n,k,t)$ denote the minimum total time to transport persons $\left\{1,\ldots,n\right\}$ with crossing times $1<2<\cdots<n$, where the center of the star graph $\mathcal{S}(k)$ is $r$ and the leaves are $u_1,\ldots, u_k$. Say the $n$ persons and the $t$ torches start at $r$ and each torch-move carries one or two people along one edge and every parallel step uses disjoint people and disjoint star edges. Then, \begin{align}\label{eqn: star graph lower and upper bounds}
\frac{n^2}{4\min \{k,t\}}+O(n)\le T(n,k,t)\le \frac{n^2}{2}+O(n).
\end{align}
In particular, \[T(n,k,t)\ge sn-ms(s-1)\quad\text{where }m=\min\{k,t\} \text{ and }s=\left\lceil \frac{n}{2m}\right\rceil\]
and equality holds if $\lceil \frac{n}{2}\rceil\le t\le n $.
\end{theorem}

Before discussing the proof of Theorem \ref{theorem: bridge and torch for star graphs}, we look at some important cases. First, if $m=\min\{k,t\}$ is large enough such that $m\ge \lceil \frac{n}{2}\rceil$, then $s=1$ and $T(n,k,t)=n$. On the other hand, if $k=1$ and $t\ge \lceil\frac{n}{2}\rceil$, then $m=1$ and $s=\lceil\frac{n}{2}\rceil$, which implies \[T(n,1,t)=\left\lfloor \frac{(n+1)^2}{4}\right\rfloor.\]
By considering the expression in (\ref{eqn:formula for T(n)}), we see that for all $n\ge 2$, \[T(n,1,t)=\left\lfloor \frac{(n+1)^2}{4}\right\rfloor\le \left\lfloor \frac{n^2+12n-20}{4}\right\rfloor=T(n).\]
This makes sense because the conventional bridge and torch problem only allows one torch.

In the general setting where $t = \lceil \frac{n}{2}\rceil$, and $k \ge 2$ and $n$ vary, a query of the values in Figure \ref{fig values star tnkfloor c2} on the OEIS \cite{OEIS2026} yields sequences about partial sums of floors with no comments or connections to the bridge and torch problem. For example, \seqnum{A130519} states that the sequence of numbers $1,2,3,4,6,8,10,12,\ldots$ is \[\sum_{j=0}^{n+3}\left\lfloor \frac{j}{4}\right\rfloor\]
which by our discussion in Theorem \ref{theorem: bridge and torch for star graphs}, is equal to \[\left\lceil \frac {n}{4}\right\rceil\left(n-2\left(\left\lceil \frac {n}{4}\right\rceil-1\right)\right).\]
We have thus recovered the following identity:
\[\sum_{j=0}^{n}\left\lfloor \frac{j}{4}\right\rfloor=\left\lceil \frac {n-3}{4}\right\rceil\left(n-3-2\left(\left\lceil \frac {n-3}{4}\right\rceil-1\right)\right).\]
Similar identities for different $2k$ can be
found in this manner.
\begin{figure}[H]
    \centering
    \begin{tabular}{|c|c|c|}
        \hline
        $k$ & Sequence values for $n=1,\dots,20$ & Related seq. \\
        \hline
        2 & 1, 2, 3, 4, 6, 8, 10, 12, 15, 18, 21, 24, 28, 32, 36, 40, 45, 50, 55, 60
        & \seqnum{A130519} \\
        \hline
        3 & 1, 2, 3, 4, 5, 6, 8, 10, 12, 14, 16, 18, 21, 24, 27, 30, 33, 36, 40, 44
        & \seqnum{A174709} \\
        \hline
        4 & 1, 2, 3, 4, 5, 6, 7, 8, 10, 12, 14, 16, 18, 20, 22, 24, 27, 30, 33, 36
        & \seqnum{A118729} \\
        \hline
        5 & 1, 2, 3, 4, 5, 6, 7, 8, 9, 10, 12, 14, 16, 18, 20, 22, 24, 26, 28, 30
        & \seqnum{A117804} \\
        \hline
        6 & 1, 2, 3, 4, 5, 6, 7, 8, 9, 10, 11, 12, 14, 16, 18, 20, 22, 24, 26, 28
        & \seqnum{A221912} \\
        \hline
    \end{tabular}
    \caption{Values of $T(n,k,\lceil\frac{n}{2}\rceil)$
    for varying $n=1,\dots,20$ and $k=2,\dots,6$}
    \label{fig values star tnkfloor c2}
\end{figure}

We now give a proof of Theorem \ref{theorem: bridge and torch for star graphs}.
\begin{proof}
Let $m=\min \{k,t\}$. Since each parallel step must use disjoint star edges, at most one torch can traverse each edge $\{0,i\}$ in a given step for all $1\le i \le k$. Hence, regardless of how many torches are available, at most $k$ torches can move out of (or into) the center $0$ simultaneously. On the other hand, we only have $t$ torches, so in any parallel step, we can exclude at most $m=\min\{k,t\}$ torch-moves that leave the center. 

Moreover, each bridge in $\mathcal{S}(k)$ is assumed to have capacity $2$, so in a single parallel step, we can transport at most $2m$ persons from the center to the leaves. As such, to transport all $n$ people from $r$ to $\{1,\dots,k\}$, we need at least $s=\lceil \frac{n}{2m}\rceil$ parallel steps that contain at least one outward move from $r$.

We first prove the lower bound. In the first outward step (the first step in which some person leaves $r$), at most $2m$ people can leave $r$, so at least one of the people $\{n-2m+1,\dots,n\}$ must move outward in that step. Hence the duration of that step is at least $n$. Similarly, after $j-1$ outward steps, at most $2m(j-1)$ people have been transported to leaves. Thus, before the $j^\text{th}$ outward step begins, at least one of the people
\[\{n-2m(j-1),n-2m(j-1)+1,\dots,n\}\]
must still be at the center, and therefore the duration of the $j^\text{th}$ outward step is at least $n-2m(j-1)$. Summing over $j=1,2,\ldots,s$ gives
\[T(n,k,t) \ge \sum_{j=1}^{s}(n-2m(j-1))=sn-2m\sum_{j=1}^{s}(j-1)=sn-m s(s-1).\]
Since $s=\left\lceil \frac{n}{2m}\right\rceil$, this bound implies in particular
\[T(n,k,t) \ge \frac{n^2}{4m}+O(n)\]
because $s=\frac{n}{2m}+O(1)$ and hence $sn-ms(s-1)=\frac{n^2}{4m}+O(n)$.

As for the upper bound, we give a concrete (but not necessarily optimal) schedule of total time $\frac{n^2}{2}+O(n)$ that works for all $k,t\ge 1$. Fix any leaf, say $u_1$. Use a single torch to shuttle people to leaf $1$ using person 1 as the returner: for each $i=2,3,\dots,n$, we send person 1 and $i$ from the center $r$ to the leaf $u_1$ with cost $i$, and then person 1 from $u_1$ to the center $r$ with cost 1. After $i=n$, we finally move person $1$ back to $u_1$ with cost 1. This yields the total time \[\sum_{i=2}^{n}i+(n-2)\cdot 1+1=\frac{n^2+3n-4}{2}=\frac{n^2}{2}+O(n).\]
This proves the upper bound as in (\ref{eqn: star graph lower and upper bounds}).

As for equality, we now assume that $t\ge \lceil n/2\rceil$. Then, we can transport all people in one-shot outward-only batches without ever returning a torch to the center. Indeed, in one parallel step, each outward torch-movement can carry two persons so with $t$ torches, we can move up to $2t\ge n$ people outward in that single step. If $k\ge t$, we assign each torch to a distinct edge; if $k<t$, then $m=\min\{k,t\}=k$ and we instead use the $k$ available edges (still each at most once per step) to move up to $2k=2m$ people outward per step. In either case, the number of outward moves per step is bounded by $m$, so we can move $2m$ people outward per step.

More concretely, let $s=\lceil\frac{n}{2m}\rceil$. We partition the set of people into $s$ groups, each of size at most $2m$, and in step $j$ move exactly that group from $r$ to \textbf{what} leaves using $m$ torches (some torches may carry only one person if needed). Arrange the groups in decreasing order of their largest label, so that the maximum crossing time in step $j$ is exactly $n-2m(j-1)$. Then, the total time of this schedule is
\[\sum_{j=1}^{s}(n-2m(j-1))=sn-ms(s-1).\]
Since we already proved the lower bound earlier, this schedule is optimal and hence equality holds. That is, $T(n,k,t)=sn-ms(s-1)$.
\end{proof}

\section{Acknowledgments}
We would like to thank our friends Loh Wei Xuan Ryan and Malcolm Tan Jun Xi from the National University of Singapore for looking through this preprint. The second
author thanks Chuanjie Duanmu from River Valley High School for providing assistance
with programming in SageMath \cite{sagemath}. Moreover, we would also like to thank Beno\^{i}t Corsini from the National University of Singapore for providing invaluable guidance.

\printbibliography

\end{document}